\documentclass[12pt,leqno,amscd, amsfonts, amssymb,pstricks,verbatim]{amsart}
\usepackage{amsfonts}
\usepackage{mathrsfs}
\usepackage{amsmath}
\usepackage{amssymb}
\usepackage{hyperref}

\def\a{\alpha}

\def\b{\beta}

\def\Z{\mathbb{Z}}
\def\N{\mathbb{N}}

\def\C{\mathbb{C}}

\numberwithin{equation}{section}
\newtheorem{theo}{Theorem}[section]
\newtheorem{defi}[theo]{Definition}

\newtheorem{prop}[theo]{Proposition}

\newtheorem{rema}[theo]{Remark}
\newtheorem{case}{Case}

\newtheorem{subcase}{Subcase}

\allowdisplaybreaks

\begin{document}

\title[Non-weight modules over algebras related to the Virasoro algebra]{Non-weight modules over algebras related to the Virasoro algebra}

\author{Qiu-Fan Chen and Yu-Feng Yao}

\address{Department of Mathematics, Shanghai Maritime University,
 Shanghai, 201306, China.}\email{chenqf@shmtu.edu.cn}

\address{Department of Mathematics, Shanghai Maritime University,
 Shanghai, 201306, China.}\email{yfyao@shmtu.edu.cn}



\thanks{This work is supported by National Natural Science Foundation of China (Grant Nos. 11771279, 11371278, 11431010, 11571008 and 11671138) and
Natural Science Foundation of Shanghai (Grant No. 16ZR1415000).}

\begin{abstract}
In this paper, we study a class of non-weight modules over two kinds of algebras related to the Virasoro algebra, i.e., the loop-Virasoro algebras $\mathfrak{L}$ and a class of Block type Lie algebras $\mathfrak{B(q)}$, where $q$ is a nonzero complex number. We determine those modules whose restriction to the Cartan subalgebra (modulo center) are free of rank one. We also provide a sufficient and necessary condition for such modules to be simple, and determine their isomorphism classes. Moreover,  we obtain the simplicity of modules over loop-Virasoro algebras by taking tensor products of some irreducible modules mentioned above with irreducible highest weight modules or Whittaker modules.
\end{abstract}

\maketitle

\section{Introduction}

Throughout the paper, we denote by $\C ,\,\Z,\,\C^*,\,\Z^*,\,\Z_+,\,\N$ the sets of complex numbers, integers, nonzero complex numbers, nonzero integers, nonnegative integers and positive integers, respectively.  $\C[t]$ is used to denote the polynomial algebra over $\C$. All vector spaces (resp. (Lie) algebras) are over $\C$. For a Lie algebra $\mathfrak{g}$, we use $U(\mathfrak{g})$ to denote the universal enveloping algebra of $\mathfrak{g}$.

The Virasoro algebra, denoted by $\mathfrak{V}$, is an infinite dimensional Lie algebra over $\C$ with basis $\{L_i, C\mid i\in\Z\}$ and defining relations
\begin{equation*}
\aligned
&[L_i,L_j]=(j-i)L_{i+j}+\delta_{i+j,0}\frac{i^3-i}{12}C,\,\,\,\forall\,i,j\in\mathbb{Z},\\
&[L_i,C]=0,\,\,\,\forall\,i\in\mathbb{Z},
\endaligned
\end{equation*}
which is the universal central extension of the so-called infinite dimensional Witt algebra of rank one. The Virasoro algebra is one of the most important
Lie algebras in both mathematics and mathematical physics (cf. \cite{KR}). The representation theory of the Virasoro algebra has been extensively studied.
In \cite{M}, O. Mathieu classified all simple Harish-Chandra modules,  which was conjectured by Kac \cite{K}. In recent years, many authors constructed various simple non-Harish-Chandra modules and simple non-weight modules(cf.~\cite{BM,LGZ,MW,MZ,TZ1,LLZ,LZ,TZ,CG2,LZ2,CH}). In particular, the authors in \cite{TZ} constructed  a class of $\mathfrak{V}$-modules $\Omega(\lambda, \a)(\lambda\in\C^*,\a\in\C)$ that are free of rank one when restricted to the Cartan subalgebra. Explicitly, the  $\mathfrak{V}$-module structure on $\Omega(\lambda, \a)=\C[t]$ is given by
\begin{equation}\label{intro}
\aligned
&L_i\cdot f(t)=\lambda^i(t-i\a)f(t-i),
\,\, \forall \, \, i\in\Z, \ f(t)\in\C[t],\\
&C\cdot f(t)=0, \forall \, f(t)\in\C[t].
\endaligned
\end{equation}
Actually, this class of modules were first introduced and studied in \cite{LZ} as quotient modules of fraction $\mathfrak{V}$-module. Recently, this kind of non-weight modules, which many authors call $U(\mathfrak{h})$-free modules, have been extensively studied.

The notation of $U(\mathfrak{h})$-free modules was first introduced by J. Nilsson \cite{N} for the simple Lie algebra $\mathfrak{sl_{n+1}}$. The idea originated in the attempt to understand whether the general setup for study of Whittaker modules proposed in \cite{BM} can be used to construct some explicit families of simple $\mathfrak{sl_{n+1}}$-modules. In this paper and a subsequent paper \cite{N1}, Nilsson showed that a finite dimensional simple Lie algebra has nontrivial $U(\mathfrak{h})$-free modules if and only if it is of type $A$ or $C$. Furthermore, the $U(\mathfrak{h})$-free modules of rank one for the Kac-Moody Lie algebras were determined in \cite{CTZ}. They proved that there are no nontrivial $U(\mathfrak{h})$-free modules of rank one for affine type and indefinite type. And the idea  was  exploited and generalized to consider modules over infinite  dimensional Lie algebras, such as the Witt algebras of all ranks \cite{TZ},  Heisenberg-Virasoro algebra and $W(2,2)$ algebra  \cite{CG},  simple finite-dimensional Lie superalgebras \cite{CZ}, the algebras $\mathfrak{V}ir(a,b)$  \cite{HCS}, the Lie algebras related to the Virasoro algebra \cite{CC} and so on. The aim of this paper is to classify such modules for the loop-Virasoro algebras and a class of Block type Lie algebras.

This paper is organized as follows. In Section 2, we construct a class of non-weight modules over the loop-Virasoro algebra. The simplicity and isomorphism classes of these modules are explicitly determined. Moreover, we obtain the simplicity of modules over the loop-Virasoro algebra by taking tensor products of some irreducible modules mentioned above with irreducible highest weight modules or Whittaker modules. Section 3 is devoted to classifying the modules whose restriction to the Cartan subalgebra are free of rank one over a class of Block type Lie algebras. We also provide a necessary and sufficient condition for such modules to be simple.

\section{Modules over the loop-Virasoro algebra}
The loop-Virasoro algebra $\mathfrak{L}$ is the Lie algebra that is the tensor product of the Virasoro algebra  $\mathfrak{V}$ and the Laurent polynomial algebra $\C[t^{\pm1}]$, i.e., $\mathfrak{L}=\mathfrak{V}\otimes \C[t^{\pm1}]$
with $\C$-basis $\{L_i\otimes t^j,C\otimes t^j\mid (i,j)\in\Z^2\}$ and defining relations
\begin{equation*}
\aligned
&[L_{i}\otimes t^{j},L_{k}\otimes t^{l}]=(k-i)L_{i+k}\otimes t^{j+l}+\delta_{i+k,0}\frac{i^3-i}{12}(C\otimes t^{j+l}),\\
&[L_i\otimes t^j,C\otimes t^{l}]=0.
\endaligned
\end{equation*}
We see that $\mathfrak{L}$ has a copy of Virasoro algebra which is $\mathfrak{V}\otimes1$. For convenience, we simply write $L_{i,j}=L_i\otimes t^j$ and $C_i=C\otimes t^i$. It is obvious that $\mathfrak{L}$ has the natural $\Z$-grading \begin{equation*}\mathfrak{L}=\oplus_{i\in\Z}\mathfrak{L}_i,\ \ \ \mathfrak{L}_i={\rm span_{\C}}\{L_{i,j},\delta_{i,0}C_j\mid j\in\Z\}.\end{equation*} Note that  $\mathfrak{L}_0$ is an infinite dimensional abelian subalgebra of $\mathfrak{L}$, and that $C\otimes \C[t^{\pm1}]$ is the center of $\mathfrak{L}$. The Cartan subalgebra  (modulo center) of $\mathfrak{L}$ is spanned by $L_{0,0}$. Various classes of representation of the loop-Virasoro algebra were studied and classified in \cite{GLZ,LG}.
In this section, we will determine the $\mathfrak{L}$-modules which are free of rank $1$ when regarded as $\C L_{0,0}$-modules.

\begin{defi}\label{defi2.1}\rm  For $\lambda,\mu\in\C^*,\,\,\a\in\C$, define the action of $\mathfrak{L}$  on $\Omega(\lambda, \mu,\a):=\C[t]$ as follows:
\begin{equation}\label{L-action}
\aligned
&L_{i,j}\cdot f(t)=\lambda^{i-j}\mu^j(t-i\a)f(t-i),\\
&C_i\cdot f(t)=0,
\endaligned
\end{equation}where $f(t)\in\C[t], (i,j)\in\Z^2$.
\end{defi}
\begin{prop}
$\Omega(\lambda, \mu,\a)$ is an $\mathfrak{L}$-module under the action given in Definition \ref{defi2.1}.
\end{prop}
\begin{proof}
According to the above definition, for any $f(t)\in\C[t]$ and $(i,j,k,l)\in\Z^4$, we have the following straightforward computation by (\ref{L-action})
\begin{eqnarray*}\label{a4}
\!\!\!\!\!\!
L_{i,j}\cdot L_{k,l}\cdot f(t)&\!\!\!=\!\!\!&
\lambda^{k-l}\mu^lL_{i,j}\cdot\big((t-k\alpha) f(t-k)\big)\nonumber\\
&\!\!\!=\!\!\!&
\lambda^{i+k-j-l}\mu^{j+l}(t-i\alpha)(t-k\alpha-i)f(t-i-k),
\end{eqnarray*}
from which we get
\begin{eqnarray*}
[L_{i,j},L_{k,l}]\cdot f(t)&=&\big((k-i)L_{i+k,j+l}+\delta_{i+k,0}\frac{i^3-i}{12}C_{j+l}\big)\cdot f(t)\nonumber\\
&=&(k-i)L_{i+k,j+l}\cdot f(t)\nonumber\\
&=&
\lambda^{i+k-j-l}\mu^{j+l}(k-i)(t-i\alpha-k\alpha)f(t-i-k)\nonumber\\
&=&\lambda^{i+k-j-l}\mu^{j+l}(t-i\alpha)(t-k\alpha-i)f(t-i-k)\nonumber\\
&&-\lambda^{i+k-j-l}\mu^{j+l}(t-k\alpha)(t-i\alpha-k)f(t-i-k)\nonumber\\
&=&L_{i,j}\cdot L_{k,l}\cdot f(t)-L_{k,l}\cdot L_{i,j}\cdot f(t).\nonumber\\
\end{eqnarray*}
Moreover, it directly follows from (\ref{L-action}) that
$$[L_{i,j}, C_k]\cdot f(t)=L_{i,j}\cdot C_k\cdot f(t)-C_k\cdot L_{i,j}\cdot f(t)=0, \,\,\forall\,(i,j,k)\in\mathbb{Z}^3, f(t)\in\mathbb{C}[t].$$
We completes the proof.
\end{proof}
The following result determines the isomorphism classes of the $\mathfrak{L}$-modules $\Omega(\lambda,\mu,\a)$.
\begin{prop} Let $\lambda,\lambda^{\prime},\mu, \mu^{\prime}\in\mathbb{C}^*,\a, \a^{\prime}\in\mathbb{C}$. Then the $\mathfrak{L}$-modules $\Omega(\lambda,\mu,\a)$ and $\Omega(\lambda^{\prime},\mu^{\prime},\a^{\prime})$ are isomorphic if and only if  $(\lambda,\mu,\a)=(\lambda^{\prime},\mu^{\prime},\a^{\prime}).$
\end{prop}
\begin{proof} It suffices to show the necessary part. Suppose $\varphi:\Omega(\lambda,\mu,\a)\to\Omega(\lambda^{\prime},\mu^{\prime},\a^{\prime})$ is an  isomorphism of $\mathfrak{L}$-modules with the inverse $\varphi^{-1}$. Regard $\Omega(\lambda,\mu,\a)$ and $\Omega(\lambda^{\prime},\mu^{\prime},\a^{\prime})$ as $\mathfrak{V}$-modules, we get $(\lambda,\a)=(\lambda^{\prime},\a^{\prime})$ by \cite{LZ}.  To complete the proof, we only need to show $\mu=\mu^{\prime}$. For that, let $f(t)\in\C[t]$,  we have \begin{equation*}\varphi(f(t))=\varphi(f(L_{0,0})\cdot1)=f(L_{0,0})\cdot\varphi(1)=f(t)\varphi(1).\end{equation*}
Similarly,\begin{equation*}\varphi^{-1}(f(t))= f(t)\varphi^{-1}(1).\end{equation*}
Hence, \begin{equation*}1=\varphi^{-1}(\varphi(1))=\varphi(1)\varphi^{-1}(1),\end{equation*}
which implies $\varphi(1)\in\C^*$. Combining this with $\a=\a^{\prime}$ and
\begin{equation*}\mu(t-\a)\varphi(1)=\varphi(L_{1,1}\cdot1)=L_{1,1}\cdot\varphi(1)=\mu^{\prime}(t-\a^{\prime})\varphi(1),\end{equation*}
we konw that $\mu=\mu^{\prime}$.
\end{proof}

We are now in the position to present the following main result of this section.
\begin{theo}\label{theo123456} Let $M$ be an $\mathfrak{L}$-module such that it is free of rank one as a $U(\C L_{0,0})$-module. Then $M\cong\Omega(\lambda, \mu, \a)$ for some $\lambda,\mu\in\C^*$ and $\a\in\C$. Moreover, $M$ is simple if and only if $M\cong\Omega(\lambda, \mu,\a)$ for some $\lambda,\mu,\a\in\C^*$.
\end{theo}
\begin{proof} Viewed as the Virasoro-module, we have $M\cong \Omega(\lambda, \a)$  defined in \eqref{intro}, that is,
$$L_{i,0}\cdot f(t)=\lambda^i(t-i\a)f(t-i),\ \ \ \ C_0\cdot f(t)=0, \ \ \ \mbox{for}\ \ \ i\in\Z, $$
where $f(t)\in\C[t]$ and $\lambda\in\C^*,\, \a\in\C$. For $(i,j)\in\Z\times\Z^*$, now we consider the actions of $L_{i,j}$ (resp. $C_j$) on $M$. They are completely determined by the actions of $L_{i,j}$ and $C_j$ on $1\in M$. Explicitly, for any polynomial $f(t)\in \C[t]$, we have
\begin{equation*}
\aligned
&L_{i,j}\cdot f(t)=L_{i,j}\big(f(L_{0,0})(1)\big)=f(t-i)F_{i,j}(t),\\
&C_j\cdot f(t)=f(t)c_j(t),
\endaligned
\end{equation*}
where $F_{i,j}(t):=L_{i,j}\cdot1, c_j(t)=C_j\cdot1$. From $[L_{i,0},L_{i,i}]\cdot 1=0$, we know that \begin{equation}\label{sx}
(t-i\a)F_{i,i}(t-i)=(t-i\a-i)F_{i,i}(t),\ \ \ \ \forall i\in\Z.
\end{equation}
This shows that $t-i\a$ divides $F_{i,i}(t)$. These entail us to assume that $F_{i,i}(t)=(t-i\a)f_{i}(t)$ for some $f_i(t)\in \C[t]$. Inserting this into \eqref{sx} gives $f_{i}(t)=f_i(t-i)$, which in turn requires $f_{i}(t)=\mu_i$ for some $\mu_i\in\C$ with $\mu_0=1$. Thus,\begin{eqnarray*}
F_{i,i}(t)=(t-i\a)\mu_{i},\ \ \ \forall i\in\Z.
\end{eqnarray*}
Using
\begin{eqnarray*}
[L_{1,1},L_{i,i}]\cdot 1=(i-1)L_{i+1,i+1}\cdot 1\quad {\rm and}\quad [L_{-1,-1},L_{i,i}]\cdot 1=(i+1)L_{i-1,i-1}\cdot 1,
\end{eqnarray*}
we obtain
\begin{eqnarray*}
(i-1)\mu_{i+1}=(i-1)\mu_i\mu_1\quad {\rm and}\quad (i+1)\mu_{i-1}=(i+1)\mu_i\mu_{-1}.
\end{eqnarray*}
 It follows from the above recurrence relations that $\mu_i=\mu^i$ ($\mu:=\mu_1\in\C^*$) for all $i\in\Z$. Hence, $$F_{i,i}(t)=\mu^i(t-i\a),\ \ \ \ \forall i\in\Z.$$ By $
[L_{i,0},L_{0,i}]\cdot 1=(-i)L_{i,i}\cdot1$, we obtain
\begin{equation*}
\lambda^i(t-i\a)\big(F_{0,i}(t-i)-F_{0,i}(t)\big)=-i\mu^i(t-i\a).
\end{equation*}
It follows that
\begin{equation}\label{ws}
F_{0,i}(t)=\lambda^{-i}\mu^it+e_i \mbox{ \ for some \ } e_i\in\C, \ \ \ \ \forall i\in\Z.
\end{equation}
Combining this with  $[L_{i,0},L_{0,j}]\cdot 1=(-i)L_{i,j}\cdot1$ gives \begin{equation*}
F_{i,j}(t)=\lambda^{i-j}\mu^j(t-i\a),\ \ \ \ \forall i\in\Z^*.
\end{equation*}
From
\[
[L_{1,0},L_{-1,i}]\cdot 1=(-2)L_{0,i}\cdot1,
\]
we obtain $e_i=0$. Thus, \eqref{ws} is simply written as $F_{0,i}(t)=\lambda^{-i}\mu^it$  for any $i\in\Z$. Consequently,
\begin{equation*}
F_{i,j}(t)=\lambda^{i-j}\mu^j(t-i\a),\ \ \ \ \forall (i,j)\in\Z^2.
\end{equation*}
Hence,
$$L_{i,j}\cdot f(t)=\lambda^{i-j}\mu^j(t-i\a)f(t-i),\ \ \forall (i,j)\in\Z^2, f(t)\in\mathbb{C}[t].$$
Moreover, from the following equality
$$[L_{2,j},L_{-2,0}]\cdot 1=-4L_{0,j}\cdot 1+\frac{1}{2}C_j\cdot 1,\,\,\forall\,j\in\mathbb{Z},$$ we immediately get $c_j(t)=0$ for any $j\in\mathbb{Z}$. The above discussion implies that $M\cong\Omega(\lambda,\mu,\a)$.

Furthermore, if $M$ is a simple $\mathfrak{L}$-module, then it is obvious simple as a $\mathfrak{V}$-module. It follows directly from \cite{LZ} that $\alpha\neq 0$. On the other hand, if $\alpha=0$, then it is a routine to check that $t\Omega(\lambda,\mu,0)$ is an $\mathfrak{L}$-submodule of $\Omega(\lambda,\mu,0)$.
\end{proof}

\begin{rema}
For any $\lambda,\mu\in\mathbb{C}^*$, it follows from \cite{LZ, CG2} that $\Omega(\lambda, \mu, 0)$ has the following unique composition series
$$\Omega(\lambda, \mu, 0)\supset t\Omega(\lambda, \mu, 0)\supset 0,$$
where $\Omega(\lambda, \mu, 0)/ \big(t\Omega(\lambda, \mu, 0))$ is the one-dimensional trivial $\mathfrak{L}$-module, and $t\Omega(\lambda, \mu, 0)\cong
\Omega(\lambda, \mu, 1)$.
\end{rema}

The following result asserts the simplicity of $\mathfrak{L}$-modules by taking tensor products of some irreducible modules in Theorem \ref{theo123456} with irreducible highest weight modules or Whittaker modules.
\begin{theo} Let $m\in\N,\,\lambda_i,\mu_i,\a_i\in\C^*$ for $i=1,2,\ldots,m$ with the $\lambda_i$ pairwise distinct. Let $V$ be a highest weight module or Whittaker module over $\mathfrak{L}$. Then the tensor product $(\otimes_{i=1}^{m}\Omega(\lambda_i,\mu_i,\a_i))\otimes V$ is an irreducible $\mathfrak{L}$-module. Especially, $\otimes_{i=1}^{m}\Omega(\lambda_i,\mu_i,\a_i)$ is an irreducible $\mathfrak{L}$-module.
\end{theo}
\begin{proof}
The proof is similar to that of \cite[Theorem 1]{TZ1}. We omit the details.
\end{proof}

\section{Modules over Block type Lie algebras}
For $a\in\N$, we denote $\delta_{a,\N}=1$ if $a\in\N$, and $0$ otherwise. For any positive integer $i$, we use $N_{\geq i}$ to denote the positive integer greater than or equal to $i$. For any nonzero complex number $q$, the Block type Lie algebra $\widehat{\mathfrak{B}(q)}$ has a basis $\{L_{m,i}, C\mid (m,i)\in\Z\times\Z_+\}$ over $\C$ subject to the following relations
\begin{equation*}
\aligned
&[L_{m,i},L_{n,j}]=\big(n(i+q)-m(j+q)\big)L_{m+n,i+j}+\delta_{m+n,0}\delta_{i+j,0}\frac{m^3-m}{12}C,\\
&[C,L_{m,i}]=0.
\endaligned
\end{equation*}
The Lie algebra $\widehat{\mathfrak{B}(q)}$ is in fact a subalgebra of some special case of generalized Block algebras studied in \cite{DZ}. It is also a half part of the Block type algebras studied in \cite{LG1} and $\widehat{\mathfrak{B}(1)}$ is the Block type Lie algebra considered in \cite{WT}. One sees that the center of $\widehat{\mathfrak{B}(q)}$ is $Z(\widehat{\mathfrak{B}(q)})=\C C\oplus\C\delta_{-q,\N}L_{0,-q}$. In this section, we concentrate on the algebra $$\mathfrak{B}(q):=[\widehat{\mathfrak{B}(q)},\widehat{\mathfrak{B}(q)}]$$ with basis $$\{L_{m,i},C\,|\,(m,i)\in\Z\times\Z_+\}.$$
Here and below, when $-2q\in\mathbb{N}$, we abuse the notation by writing $\Z\times \Z_+$ rather than $(\Z\times\Z_+)\setminus\{(0, -2q)\}$ (we will make sure that this abuse of notation will not create any confusion). The Lie algebra $\mathfrak{B(q)}$ is interesting in the sense that it contains the following subalgebra
\begin{equation*}{\rm span_{\C}}\{q^{-1}L_{m,0},C\mid\,m\in\Z\},\end{equation*}
which is isomorphic to the well-known Virasoro algebra. The Lie algebra  $\mathfrak{B(q)}$ is interesting to us in another aspect that it also contains many important subquotient algebras $\mathfrak{B}(q)_{k,l}$ for $l\geq k\geq0$, where
\begin{equation*}\mathfrak{B}(q)_{k,l}=\mathfrak{B}(q)_{k}/\mathfrak{B}(q)_{l+1},\ \ \ \mathfrak{B}(q)_{k}={\rm span_{\C}}\{L_{m,i}\mid m\in\Z,i\geq k\}.\end{equation*}
For instance, $\mathfrak{B}(q)_{0,0}$ is the (centerless) Virasoro algebra (thus the Virasoro algebra is both a subalgebra and a quotient algebra of $\mathfrak{B}(q)$), $\mathfrak{B}(-1)_{0,1}$ is the (centerless) twisted Heisenberg-Virasoro algebra. Note that the Cartan subalgebra of $\mathfrak{B}(q)$ is spanned by $L_{0,0}$. The representation theory of $\mathfrak{B}(q)$  has been extensively studied by many authors. For example, the authors in \cite{WT,SXX1,SXX2} presented a classification of the irreducible Harish-Chandra modules over $\mathfrak{B}(q)$. In \cite{CG1} and \cite{XZ}, the authors classified the unitary Harish-Chandra modules over $\mathfrak{B}(q)$. In this section, we will determine the $\mathfrak{B}(q)$-modules which are free of rank one when regarded as
$\C L_{0,0}$-modules.

If $q\neq-1$, the $\mathfrak{V}$-module (i.e.,$\mathfrak{B}(q)_{0,0}$-module) $\Omega(\lambda, \a)$  becomes a $\mathfrak{B}(q)$-module by setting $$\mathfrak{B}(q)_{1}\cdot \Omega(\lambda, \a)=0.$$ The resulting $\mathfrak{B}(q)$-module will  be also denoted by $\Omega(\lambda, \a)$ for brevity. Similarly, the twisted Heisenberg-Virasoro-module
(i.e., $\mathfrak{B}(-1)_{0,1}$-module) $\Omega(\lambda, \a,\beta)$ constructed in \cite{CG} can be extended to a $\mathfrak{B}(-1)$-module by setting $$\mathfrak{B}(-1)_{2}\cdot \Omega(\lambda, \a,\beta)=0.$$ The resulting $\mathfrak{B}(q)$-module will be also denoted by $\Omega(\lambda,\a,\b)$ for brevity. In a uniform way, the $\mathfrak{B}(q)$-module structure on $\Omega(\lambda, \a)$ and $\Omega(\lambda,\a,\b)$ is given by
\begin{equation}\label{aaa}
\aligned
\frac{}{}&L_{m,i}\cdot f(t)=\lambda^m\big(\delta_{i,0}(t-mq\a)+\delta_{q,-1}\delta_{i,1}\beta\big)f(t-mq),\\
&C\cdot f(t)=0,
\endaligned
\end{equation}where $f(t)\in\C[t], (m,i)\in\Z\times \Z_+$.
\begin{theo}\label{theo1234} Let $M$ be a $\mathfrak{B}(q)$-module such that it is free of rank one as a $U(\C L_{0,0})$-module.
Then there exist some $\lambda\in\C^*$ and $\a,\b\in\C$ such that
\begin{equation}\label{two classes}
M\cong\begin{cases}
\Omega(\lambda, \a),&\text{if}\,\, q\neq -1,\cr
\Omega(\lambda, \a,\b),&\text{if}\,\,q=-1.
\end{cases}
\end{equation}
Moreover, $M$ is simple if and only if $\a\neq 0$ or $\beta\neq 0$ in \eqref{two classes}.
\end{theo}
\begin{proof} Regarded as the Virasoro-module, we have $M\cong \Omega(\lambda, \a)=\C[t]$, that is, $$C\cdot f(t)=0,\ \ \ \ L_{m,0}\cdot f(t)=\lambda^m(t-mq\a)f(t-mq), \ \ \ \ \forall m\in\Z. $$
For $(m,i)\in\Z\times\N$, now we consider the actions of $L_{m,i}$ on $M$, which  are completely determined by the actions of $L_{m,i}$ on $1\in M$. Explicitly, for any polynomial $f(t)\in \C[t]$, we have
\begin{equation*}
L_{m,i}\cdot f(t)=f(t-mq)H_{m,i}(t),
\end{equation*}
where $H_{m,i}(t):=L_{m,i}\cdot1$.

\begin{case}$q=-\frac{1}{2}$. \end{case}

In this case, from $[L_{-m,0},L_{m,1}]\cdot1=0$, we obtain \begin{equation*}
(t-\frac{m}{2}\a)H_{m,1}(t-\frac{m}{2})=(t-\frac{m}{2}\a+\frac{m}{2})H_{m,1}(t),\ \ \ \ \forall m\in\Z^*,
\end{equation*}
which implies $H_{m,1}(t)=0$ for any $m\in\Z^*$. Combining  this with $[L_{0,i-1},L_{m,1}]\cdot1=m(i-\frac{3}{2})L_{m,i}\cdot1$,
we see that $H_{m,i}(t)=0$ for any $(m,i)\in\Z^*\times\N_{\geq3}$. Moreover, since
$[L_{-1,i},L_{1,0}]\cdot1=(i-1)L_{0,i}\cdot1$, it follows that $H_{0,i}(t)=0$ for any $i\in\N_{\geq 2}$. Now using
\begin{eqnarray*}
[L_{1,1},L_{m-1,1}]\cdot1=(\frac{m}{2}-1)L_{m,2}\cdot1\quad {\rm and}\quad [L_{-1,1},L_{m+1,1}]\cdot1=(\frac{m}{2}+1)L_{m,2}\cdot1,
\end{eqnarray*}we get $H_{m,2}(t)=0$ for any $m\in\Z$. Putting our observation together, we know that $M\cong\Omega(\lambda, \a)$.

\begin{case} $q=-1$. \end{case}

In this case, the subalgebra generated by $\{L_{m,0},L_{m,1}\mid m\in\Z\}$ is isomorphic to the twisted Heisenberg-Virasoro algebra. Following from \cite[Theorem 2]{CG}, we have
\begin{equation*}
H_{m,1}(t)=\lambda^m\beta,\ \ \ \ \forall m\in\Z,
\end{equation*}
for some $\beta\in\C$. By $[L_{m,2},L_{-m,0}]\cdot1=0$,
we get
\begin{equation*}
(t-m\a+m)H_{m,2}(t)=(t-m\a)H_{m,2}(t-m), \ \ \ \ \forall m\in\Z,
\end{equation*}
from which we see that \begin{equation}\label{w99}H_{m,2}(t)=0, \ \ \ \ m\in\Z^*.\end{equation} From
$[L_{m,i},L_{1,1}]=(i-1)L_{m+1,i+1},\, (m,i)\in\Z\times\N_{\geq2}$, we know that
\begin{equation}\label{bbn}
\lambda \beta\big(H_{m,i}(t)-H_{m,i}(t+1)\big)=(i-1)H_{m+1,i+1}(t).
%&\lambda^2 h\big(H_{m-1,i}(t)-H_{m-1,i}(t+2)\big)=2(i-1)H_{m+1,i+1}(t),
\end{equation}
If $\beta=0$, then the above formula together with \eqref{w99} gives that $H_{m,i}(t)=0$ for any $(m,i)\in\Z\times\N_{\geq2}$.
If $\beta\neq0$, it follows from \eqref{w99} and \eqref{bbn} that $H_{m,i}(t)=0$ for any $(m,i)\in\Z\times\N_{\geq2}\setminus\{(i,i+2)\mid i\in\mathbb{N}\}$. Using $[L_{-1,2},L_{i+1,i}]\cdot 1=2iL_{i,i+2}\cdot 1$, we get $H_{i,i+2}(t)=0$ for any $i\in\mathbb{N}$.
Consequently, the subalgebra $\mathfrak{B}(q)_2$ vanishes on $M$. Then $M$ is isomorphic to $\Omega(\lambda, \a,\b)$.

\begin{case}$q\neq-\frac{1}{2},\,-1$. \end{case}

Before  the proof of this case, we present some formulae here, which will be used to do  calculations in this following. For any $(m,i)\in\Z\times\Z_+$, we have \begin{eqnarray*}
\!\!\!\!\!\!
m(i+q)L_{m,i}\cdot 1&\!\!\!=\!\!\!&[L_{0,i},L_{m,0}]\cdot 1\nonumber\\
&\!\!\!=\!\!\!&
L_{0,i}\cdot L_{m,0}\cdot 1-L_{m,0}\cdot L_{0,i}\cdot1\nonumber\\
&\!\!\!=\!\!\!&  L_{0,i}\cdot\big(\lambda^m(t-mq\a)\big)-L_{m,0}\cdot\big(H_{0,i}(t)\big),\end{eqnarray*}
which gives rise to
\begin{eqnarray}\label{w5}
\!\!\!\!\!\!
m(i+q)H_{m,i}(t)&\!\!\!=\!\!\!&\lambda^m(t-mq\a)\big(H_{0,i}(t)-H_{0,i}(t-mq)\big).\end{eqnarray}
Using this, we further have \begin{eqnarray*}
\!\!\!\!\!\!
m(i+2q)L_{0,i}\cdot 1&\!\!\!=\!\!\!&[L_{-m,0},L_{m,i}]\cdot 1\nonumber\\
&\!\!\!=\!\!\!&
L_{-m,0}\cdot L_{m,i}\cdot 1-L_{m,i}\cdot L_{-m,0}\cdot 1\nonumber\\
&\!\!\!=\!\!\!&
L_{-m,0}\cdot\big(\frac{\lambda^m(t-mq\a)(H_{0,i}(t)-
H_{0,i}(t-mq)}{m(i+q)}\big)\nonumber\\
&\!\!\!\!\!\!\!\!\!\!\!\!\!\!\!\!\!\!\!\!\!\!\!\!&
-L_{m,i}\cdot\big(\lambda^{-m}(t+mq\a)\big),\end{eqnarray*}
which shows that for any  $-q\neq i\in\Z_{+}$,
\begin{eqnarray}\label{w6v}
\!\!\!\!\!\!
m^2(i+q)(i+2q)H_{0,i}(t)&\!\!\!=\!\!\!&\big(t^2+mqt+m^2q^2\a(1-\a)\big)\big(H_{0,i}(t+mq)-H_{0,i}(t)\big)
\nonumber\\
&\!\!\!\!\!\!\!\!\!\!\!\!\!\!\!\!\!\!\!\!\!\!\!\!&
+
\big(t^2-mqt+m^2q^2\a(1-\a)\big)\big(H_{0,i}(t-mq)-H_{0,i}(t)\big).
\end{eqnarray}
Assume that
\begin{equation}\label{w87b}
H_{0,1}(t)=\sum_{k=0}^{N}a_{N-k}t^{N-k} \mbox { \ for some \ }a_{i}\in\C, 0\leq i\leq N.
\end{equation}
 Substituting this expression into \eqref{w6v} (for the case $i=1$) and comparing the coefficients of $t^{N}$ of both sides, we  obtain
\begin{equation}\label{w8b}
N(N+1)a_{N}=(1+\frac{1}{q})(2+\frac{1}{q})a_{N}.
\end{equation}

\begin{subcase} $|\frac{1}{q}|\not\in\N$. \end{subcase}

It follows from \eqref{w8b} that $a_N=0$, i.e., $H_{0,1}(t)=0$. Since
$[L_{m,i},L_{0,1}]\cdot 1=-m(1+q)L_{m,i+1}\cdot1$, it follows that $H_{m,i}(t)=0$ for any $(m,i)\in\Z^*\times\N$. While the identity $[L_{-1,i},L_{1,0}]\cdot 1=(i+2q)L_{0,i}\cdot1$
implies that $H_{0,i}(t)=0$ for $i\in\N$. That is to say that the subalgebra $\mathfrak{B}(q)_1$ vanishes on $M$. Then $M\cong\Omega(\lambda,\a)$.

\begin{subcase} $q\in\{\pm\frac{1}{n}\mid n\in\N\}\setminus \{-\frac{1}{2},-1\}$. \end{subcase}

We claim that $a_{N}=0$ in this subcase. Suppose on the contrary that $a_{N}\neq0$. One sees from  \eqref{w8b} that
\begin{equation*}
N=\left\{\begin{array}{llll} 1+\frac{1}{q},&\mbox{if  \ } q\in\{\frac{1}{n}\mid n\in\N\},\\[4pt]
-2-\frac{1}{q},&\mbox{if  \ }q\in\{-\frac{1}{n}\mid n\in\N_{\geq 3}\}.
\end{array}\right.
\end{equation*}
First we assert that $a_{N-k}=0$ for $k\geq1$ in \eqref{w87b}. The proof is given by induction on $k$.
We can write \begin{equation*}\label{ww2}
H_{0,1}(t)=a_{N}t^{N}+a_{N-1}t^{N-1}+\mbox{lower terms}.\end{equation*}
Substituting this into \eqref{w6v} (for the case $i=1$) and comparing the coefficient of $t^{N-1}$, we obtain $a_{N-1}=0$. Assume that the conclusion holds for $1\leq k \leq l-1$, that is,
\begin{equation*}
H_{0,1}(t)=a_{N}t^{N}+a_{N-l}t^{N-l}+a_{N-(l+1)}t^{N-(l+1)}+\mbox{lower terms}.\end{equation*} Also, by inserting this expression into \eqref{w6v} (for the case $i=1$)  and comparing the coefficient of $t^{N-l}$, we obtain
\begin{eqnarray}\label{w6vv}
\!\!\!\!\!\!
m^2(1+q)(1+2q)a_{N-l}&\!\!\!=\!\!\!&a_N\binom{N}{l+2}(mq)^{l+2}+a_{N-l}\binom{N-l}{2}(mq)^{2}
+a_{N-l-1}\binom{N-l-1}{1}mq\nonumber\\
&\!\!\!\!\!\!\!\!\!\!\!\!\!\!\!\!\!\!\!\!\!\!\!\!&
+mq\big(a_N\binom{N}{l+1}(mq)^{l+1}+a_{N-l}\binom{N-l}{1}mq\big)\nonumber\\
&\!\!\!\!\!\!\!\!\!\!\!\!\!\!\!\!\!\!\!\!\!\!\!\!&
+m^2q^2\a(1-\a)a_N\binom{N}{l}(mq)^{l}\nonumber\\
&\!\!\!\!\!\!\!\!\!\!\!\!\!\!\!\!\!\!\!\!\!\!\!\!&
+a_N\binom{N}{l+2}(-mq)^{l+2}+a_{N-l}\binom{N-l}{2}(-mq)^{2}
+a_{N-l-1}\binom{N-l-1}{1}(-mq)\nonumber\\
&\!\!\!\!\!\!\!\!\!\!\!\!\!\!\!\!\!\!\!\!\!\!\!\!&
-mq\big(a_N\binom{N}{l+1}(-mq)^{l+1}+a_{N-l}\binom{N-l}{1}(-mq)\big)\nonumber\\
&\!\!\!\!\!\!\!\!\!\!\!\!\!\!\!\!\!\!\!\!\!\!\!\!&
+m^2q^2\a(1-\a)a_N\binom{N}{l}(-mq)^{l}.
\end{eqnarray}
If $l$ is an odd number, we can simplify \eqref{w6vv} as
\begin{equation*}
m^2(1+q)(1+2q)a_{N-l}=2\bigg(\binom{N-l}{2}+\binom{N-l}{1}\bigg)(mq)^2a_{N-l}.
\end{equation*}
This together with $N(N+1)=(1+\frac{1}{q})(2+\frac{1}{q})$ indicates $a_{N-l}=0$. If $l$ is an even number, then rewrite \eqref{w6vv} as
\begin{equation*}
\big(N(N+1)-(N-l)(N-l+1)\big)a_{N-l}=2(mq)^l\bigg(\binom{N}{l+2}+\binom{N}{l+1}+\a(1-\a)\binom{N}{l}\bigg)a_{N}.
\end{equation*}
Since $m$ can be an arbitrary integer, it follows from the above formula that $a_{N-l}=0$, completing the induction step. Hence, we have
\begin{equation}\label{ccv}
H_{0,1}(t)=a_Nt^N.\end{equation}  Similarly, one has
\begin{equation}\label{ccv1}H_{0,i}(t)=b_it^{N_i}  \ \ \ \mbox{for}  \ \  b_i\in\C,\,\,\forall\,\,i\geq 2,\end{equation}
where\begin{equation*}
N_i=\left\{\begin{array}{llll} 1+\frac{i}{q}, &\mbox{if  \ }q\in\{\frac{1}{n}\mid n\in\N\},\\[4pt]
-2-\frac{i}{q}, &\mbox{if  \ }q\in\{-\frac{1}{n}\mid n\in\N_{\geq3}\}.
\end{array}\right.
\end{equation*}
\eqref{ccv1} in the case $i=2$ together with  \eqref{w5} gives rise to
\begin{equation}\label{ccv2}H_{1,2}(t)=\frac{\lambda b_2(t-q\a)\big(t^{N_2}-(t-q)^{N_2}\big)}{q+2}.\end{equation}
Now \eqref{ccv} along with \eqref{w5} forces
\begin{equation}\label{H11}
H_{1,1}(t)=\frac{\lambda a_N(t-q\a)\big(t^N-(t-q)^N\big)}{1+q}.
\end{equation}
Combining (\ref{ccv}), (\ref{H11}) with
$[L_{1,1},L_{0,1}]\cdot1=-(q+1)L_{1,2}\cdot1$,
we know that \begin{equation}\label{ccv44}
H_{1,2}(t)=\frac{\lambda a_N^2(t-q\a)\big(t^N-(t-q)^N\big)^2}{(1+q)^2}.
\end{equation}
By \eqref{ccv2} and \eqref{ccv44}, we have
\begin{equation}\label{key equal}
-(1+q)^2b_{2}\big(\sum_{i=1}^{N_2}\binom{N_2}{i}t^{N_2-i}(-q)^i\big)=(q+2)a_N^2\big(\sum_{j=1}^{N}\binom{N}{j}t^{N-j}(-q)^j\big)^2.
\end{equation}
By comparing the coefficients of the highest and lowest terms in both sides of (\ref{key equal}), we immediately get $N_2=2N-1$, and
\begin{eqnarray*}
(1+q)^2N_2qb_2&=&(q+2)(Nq)^2a_N^2,\\
-(-q)^{N_2}(1+q)^2b_2&=&(q+2)a_N^2q^{2N},
\end{eqnarray*}
which force $a_N=0$, a contradiction. Thus, the subalgebra $\mathfrak{B}(q)_1$ vanishes on $M$ by applying the same argument as in the Subcase $1$.
This completes the proof for the first statement. The second statement follows dirctly from \cite{LZ, CG}.
\end{proof}

\end{document}